\documentclass[12pt,leqno]{amsart}
\usepackage{amsmath, amsthm, amscd, amsfonts, amssymb,  xspace, latexsym, graphicx, fancyhdr,verbatim,enumitem, pinlabel}
\usepackage[colorlinks=true, pdfstartview=FitV, linkcolor=blue, citecolor=blue, urlcolor=blue]{hyperref}
\usepackage[linkcolor=b lue]{hyperref}
\usepackage{longtable}
\usepackage[all]{xy}

\theoremstyle{colon}

\swapnumbers
\theoremstyle{definition}
\newtheorem{definition}{Definition}[section]

\newtheorem{examples}[definition]{Examples}

\theoremstyle{plain}
\newtheorem{theorem}[definition]{Theorem}

\newtheorem{lemma}[definition]{Lemma}
\newtheorem{corollary}[definition]{Corollary}

\newlist{T-enum}{enumerate}{2}
\newlist{L-enum}{enumerate}{2}
\newlist{C-enum}{enumerate}{2}
\newlist{P-enum}{enumerate}{2}  
\newlist{Pf-enum}{enumerate}{2} 
\newlist{D-enum}{enumerate}{2}
\newlist{Ex-enum}{enumerate}{2}
\newlist{Exs-enum}{enumerate}{2}
\newlist{E-enum}{enumerate}{2}
\newlist{R-enum}{enumerate}{2}
\setlist[T-enum,1]{label=(\roman*),format=\bfseries\emph,leftmargin=*,labelindent=.1\parindent}
\setlist[T-enum,2]{label=(\alph*),format=\bfseries\emph,leftmargin=*,labelindent=.1\parindent}
\setlist[L-enum,1]{label=(\roman*),format=\bfseries\emph,leftmargin=*,labelindent=.1\parindent}
\setlist[L-enum,2]{label=(\alph*),format=\bfseries\emph,leftmargin=*,labelindent=.1\parindent}
\setlist[C-enum,1]{label=(\roman*),format=\bfseries\emph,leftmargin=*,labelindent=.1\parindent}
\setlist[C-enum,2]{label=(\alph*),format=\bfseries\emph,leftmargin=*,labelindent=.1\parindent}
\setlist[P-enum,1]{label=(\roman*),format=\bfseries\emph,leftmargin=*,labelindent=.1\parindent}
\setlist[P-enum,2]{label=(\alph*),format=\bfseries\emph,leftmargin=*,labelindent=.1\parindent}
\setlist[Pf-enum,1]{label=(\roman*), leftmargin=*,labelindent=.1\parindent}
\setlist[Pf-enum,2]{label=(\alph*), leftmargin=*,labelindent=.1\parindent}
\setlist[D-enum,1]{label=\textbf{\arabic*.},leftmargin=*,labelindent=.2\parindent}
\setlist[D-enum,2]{label=\textbf{(\alph*)},leftmargin=*,labelindent=.1\parindent}
\setlist[Ex-enum,1]{label=\textbf{\arabic*.},leftmargin=*,labelindent=.15\parindent}
\setlist[Exs-enum,1]{label=\textbf{\arabic*.},leftmargin=*,labelindent=.15\parindent}
\setlist[E-enum,1]{label=\textbf{\arabic*.},leftmargin=*,labelindent=.15\parindent}
\setlist[R-enum,1]{label=\textbf{\arabic*.},leftmargin=*,labelindent=.15\parindent}
\newcommand{\bra}{\ensuremath{\langle}}
\newcommand{\ket}{\ensuremath{\rangle}}

\newcommand{\supp}{\ensuremath{{\rm{supp}}}}
\newcommand{\card}{\ensuremath{{\rm{card}\,}}}

\newcommand{\ns}{\ensuremath{{{\rm{ns}}}}}

\newcommand{\rest}{\ensuremath{\!\!\downharpoonright\!}}

\makeatletter
\@namedef{subjclassname@2020}{\textup{2020} Mathematics Subject Classification}
\makeatother

\begin{document}
\title[Non-Standard Version of Egorov Algebra]{Non-Standard Version of Egorov Algebra of Generalized Functions}
\author[Todor D. Todorov]{Todor D. Todorov}
\address{Professor Emeritus of               
                        California Polytechnic State University,
                        San Luis Obispo, CA-93407, USA}
\email{ ttodorov@calpoly.edu}
\urladdr{https://math.calpoly.edu/todor-todorov}

\begin{abstract} We consider a non-standard version of Egorov's algebra of generalized functions, with improved properties of the generalized scalars and embedding of the Schwartz distributions compared with the original standard  Egorov's version. The embedding of distributions is similar to, but different from author's works in the past and independently done by Hans Vernaeve. 
\end{abstract}
\maketitle
\section{Introduction}\label{S: Introduction}

	 Egorov's article (Egorov~\cite{yEgorov90a}) on algebra $\mathcal G(\Omega)$ on generalized functions was published  relatively soon after the arrival of Colombeau theory of generalized functions (Colombeau~\cite{jCol84a}) and from the very beginning it was treated from the mathematical community in close comparison with Colombeau theory. One striking difference in this comparison is	the \emph{simplicity of Egorov construction}: Unlike in Colombeau construction, all representatives of the generalized functions are \emph{moderate} and the \emph{ideal is relatively simple} (even  \emph{trivial} in the case of generalized scalars). This simplicity of Egorov's construction is  particularly  advantageous when one is trying to define \emph{composition between generalized functions} or \emph{generalized functions on a manifold}.  In addition to the general theory, numerous interesting applications of the theory to partial differential equations appear in (Egorov~\cite{yEgorov90a}). In spite of all of these, Egorov's theory was mostly ignored from mathematical community dealing with non-linear theories of generalized functions (standard and non-standard alike) - for one and one reason only: the embedding of Schwartz distribution into Egorov algebra $\mathcal G(\Omega)$ is not of \emph{Colombeau type} in the sense that the product on $\mathcal G(\Omega)$, if restricted to smooth functions in $\mathcal E(\Omega)$, reduces to classical product only for constant functions. We summarize this in ${\mathbb C}\subset\mathcal D^\prime(\Omega)\subset\mathcal G(\Omega)$ (compared with $\mathcal E(\Omega)\subset\mathcal D^\prime(\Omega)\subset\mathcal G(\Omega)$ in Colombeau theory). The purpose of this article is to improve the properties of the generalized scalars and the embedding of the distributions as much as possible, while preserving the rest of the attractiveness features of Egorov approach including its simplicity. 

	The non-standard version of Egorov algebra $\widehat{^*\mathcal E}(\Omega)$ of generalized functions had been studied in the past under the notation $\mathcal A(\Omega)$ in (Todorov~\cite{tdTod96}, p. 680-684) and under the notation ${^*\mathcal C}^\infty(\Omega)|\ns(^*\Omega)$ in (Vernaeve~\cite{HVernPHD}). For convenience of the reader we give an independent presentation in Section~\ref{S: Non-Standard Version of Egorov Algebra widehat*E(Omega)}. One reason to involve non-standard analysis into non-linear theory of generalized functions (Egorov or Colombeau theories alike) is to improve the properties of generalized scalars (defined as generalized functions with zero gradient). The sets of generalized scalars in both Egorov and Colombeau algebras are rings with zero divisors. In contrast, the sets of scalars,  $^*\mathbb R$ and $^*\mathbb C$, of the algebra $\widehat{^*\mathcal E}(\Omega)$ are fields, real closed and algebraically closed, respectively. 
		
	In Section~\ref{S: Non-Standard Delta-Function} we discuss the existence of a particular non-standard delta-function in the space $^*\mathcal D(\mathbb R^d)$, slightly modifying some results in Todorov~\cite{tdTod90}-\cite{tdTod92}. Here $^*\mathcal D(\mathbb R^d)$ stands for  the non-standard extension of the space of test functions $\mathcal D(\mathbb R^d)$. 
	
		In Section~\ref{S: Embedding of Distributions into widehat*E(Omega)} we define a particular embedding $\iota_\Omega: \mathcal D^\prime(\Omega)\to \widehat{^*\mathcal E}(\Omega)$ of the space of Schwartz distributions. The embedding $\iota_\Omega$ is similar to, but different the embedding of distributions in (Vernaeve~\cite{HVernPHD}); we make a short comparison below. The properties of $\iota_\Omega$ are described in Theorem~\ref{T: Properties of the Embedding}, but the differences with the previous works are best visible in Corollary~\ref{C:  Multiplication of Classical Functions}. In short, the product in  $\widehat{^*\mathcal E}(\Omega)$ reduces to the classical (pointwise) product on the ring $\mathbb C[\Omega]$ of polynomials (not on the whole $\mathcal E(\Omega)$) and in a weaker sense on the ring $\mathcal C(\Omega)$ of continuous functions. This is an improvement relative to Egorov theory (Egorov~\cite{yEgorov90a}), where the product in Egorov's algebra $\mathcal G(\Omega)$ reduces to the classical (pointwise) product only on $\mathbb C$ (if complex numbers are treated as constant functions in $\mathcal E(\Omega)$). We should mention that in both $\mathcal G(\Omega)$ and $ \widehat{^*\mathcal E}(\Omega)$ there is one more embedding (in addition to what we discussed above) of $ \mathcal E(\Omega)$ as a differential subalgebra; in our text it appears under notation $\sigma$ (see the end of Definition~\ref{D: Non-Standard Version of Egorov Algebra} and Corollary~\ref{C:  Multiplication of Classical Functions}). 	
			
	In Section~\ref{S: Regular Algebra} we introduce a differential subalgebra $\widehat{\mathcal R}_\rho(\Omega)$ of  $\widehat{^*\mathcal E}(\Omega)$, which is similar to, but different from the algebra $\mathcal G^\infty(\Omega)$ of regular generalized functions introduced and study in (Oberguggenberger~\cite{mOber92},\,\cite{mOber06}) within (standard) Colombeau theory (Colombeau~\cite{jCol84a}).  The algebra $\widehat{\mathcal R}_\rho$  does not contain the counterexample constructed in Vernaeve~\cite{HVern21}) and in that way we remove the obstacles to developing a regularity methods in non-standard setting. 

	We shortly compare our approach  based on the algebra $\widehat{^*\mathcal E}(\Omega)$ with Hans Vernaeve's work on his algebra ${^*\mathcal E}(\Omega)|\ns(^*\Omega)$ in (Vernaeve~\cite{HVernPHD}). We should mention that Vernaeve translated his theory also in standard setting (Vernaeve~\cite{HVern03}). 
\begin{itemize}
\item The algebra $\widehat{^*\mathcal E}(\Omega)$ defined here and Vernaeve's algebra ${^*\mathcal E}(\Omega)|\ns(^*\Omega)$ (Vernaeve~\cite{HVernPHD}) are the same: $
\widehat{^*\mathcal E}(\Omega)={^*\mathcal E}(\Omega)|\ns(^*\Omega)$ and $\mu(\Omega)=\ns(^*\Omega)$. So, the scalars, $^*\mathbb C$ and $^*\mathbb R$, are also the same.  \emph{The difference is only in the embeddings} of the space of distributions.
\item Vernaeve's  embedding is of \emph{Colombeau type} in the sense that, 
the product on ${^*\mathcal E}(\Omega)|\ns(^*\Omega)$ \emph{reduces on $\mathcal E(\Omega)$ to the usual classical product between smooth functions} - very much like in Colombeau theory (Colombeau~\cite{jCol84a}) as well as in its non-standard versions (Oberguggenberger\,\&Todorov~\cite{OberTod98} and Todorov\,\&Vernaeve~\cite{TodVern08}). Unfortunately, Vernaeve's embedding is defined \emph{for convex open sets  $\Omega$ only}. Hence,  the family of spaces of distributions $\{{\mathcal D^\prime}(\Omega)\}_{\Omega\in\mathcal T^d}$  fails to be a subsheaf of $\{{^*\mathcal E}(\Omega)|\ns(^*\Omega)\}_{\Omega\in\mathcal T^d}$, where $\mathcal T^d$ stands for the usual topology on $\mathbb R^d$.  Consequently, the embedding in Vernaeve's approach does not, in general, preserve the supports of distributions. 
\item In contrast to the above, our embedding $\iota_\Omega$ is  \emph{not of Colombeau type} - the product in $\widehat{^*\mathcal E}(\Omega)$ \emph{generalizes the classical product on the ring of polynomials $\mathbb C[\Omega]$ only}, not on the whole space $\mathcal E(\Omega)$ (and also on the ring of continuous functions $\mathcal C(\Omega)$  in a weaker sense - after the restriction of the functions to $\Omega$). However, our embedding  is well-defined for \emph{any open set $\Omega$ of} $\mathbb R^d$, the family  $\{{\mathcal D^\prime}(\Omega)\}_{\Omega\in\mathcal T^d}$ is a subsheaf of $\{\widehat{^*\mathcal E}(\Omega)\}_{\Omega\in\mathcal T^d}$. Consequently, $\iota_\Omega$ preserves the support of distributions (very much like in Egorov and Colombeau algebras). Whether the ``trade-off'' is worth doing, remains to be seen.
\item \emph{The question} for defining a \emph{Colombeau type of embedding} of Schwartz's distributions into $\widehat{^*\mathcal E}(\Omega)$ or ${^*\mathcal E}(\Omega)|\ns(^*\Omega)$ \emph{for every open set} $\Omega$ of $\mathbb R^d$, which preserves the support of distributions,  \emph{remains open} (with or without the requirement on the generalized scalars to be fields).
\end{itemize}	
\section{Notations and Set-Theoretical Framework}\label{S: Notations and Set-Theoretical Framework}

\begin{itemize}
\item  If $\Omega$ is an open subset of $\mathbb R^d$,  we denote by $\mathcal C(\Omega)$ the space continuous functions from $\Omega$ to $\mathbb C$. Similarly, we write $\mathcal E(\Omega)=\mathcal C^\infty(\Omega)$, $\mathcal D(\Omega)=\mathcal C_0^\infty(\Omega)$,  $\mathcal L_{loc}(\Omega)$, $\mathcal D^\prime(\Omega)$ and $\mathcal E^\prime(\Omega)$ for the popular classes of functions and distributions. The \emph{Schwartz embedding} $\mathcal S_\Omega\!: \mathcal L_{loc}(\Omega)\mapsto\mathcal D^\prime(\Omega)$  is defined by $\langle\mathcal S_\Omega (f), \varphi\big\rangle=\int_\Omega f(x)\varphi(x)\, dx$ for all $\varphi\in\mathcal D(\Omega)$\, (Vladimirov~\cite{vVladimirov}).  In addition, we let $\mathbb C[\Omega]=\mathbb C[x_1,\dots, x_d]\,\rest\Omega$, where $\mathbb C[x_1,\dots, x_d]$ stands for the ring of polynomials in $d$-many variables with coefficients in $\mathbb C$ and $\;\rest\;$ stands for the point-wise \emph{restriction}.
\item Our framework is a $\frak c_+$-saturated ultrapower non-standard model with the set of individuals $\mathbb R$, where $\frak c=\card\, \mathbb R$. For a presentation of the topic we refer to (Lindstr\o m~\cite{tLindstrom}, Loeb\,\&Wolff~\cite{LoebWolff}) and/or (the Appendix in Todorov~\cite{tdTod96}). If $S$ is a set (in the superstructure of $\mathbb R$), we write $^*S$ for the non-standard extension of $S$. In particular, $^*\mathbb N$, $^*\mathbb R$, $^*\mathbb C$, $^*\mathbb R^d$, $^*\Omega$, $^*\mathcal L_{loc}(\Omega)$, $^*\mathcal C(\Omega)$, ${^*\mathcal D}(\Omega), {^*\mathcal E}(\Omega),  {^*\mathcal D}^\prime(\Omega)$. etc., are the non-standard extensions of $\mathbb N$, $\mathbb R$, $\mathbb C$, $\mathbb R^d$, $\Omega$, $\mathcal L_{loc}(\Omega)$, $\mathcal C(\Omega)$, $\mathcal D(\Omega), \mathcal E(\Omega),  \mathcal D^\prime(\Omega)$. etc., respectively.  Recall that $^*\mathbb R$ is $\frak c_+$--saturated \emph{real closed (non-Archimedean) field} of cardinality $\frak c_+$, which contains $\mathbb R$ as a subfield. Also, $^*\mathbb C$ is an \emph{algebraically closed field} containing $\mathbb C$ as a subfield and we have the usual connection $^*\mathbb C={^*\mathbb R}(i)$.  Notice that ${^*\mathcal E}(\Omega)$ and ${^*\mathcal D}(\Omega)$  are  \emph{differentia algebras} over the field $^*\mathbb C$. Also, $^*\mathbb C$ is a differential subring of $^*\mathcal E(\Omega)$ (if the elements of $^*\mathbb C$ are treated as \emph{constant functions}). We should mention that the functions in ${^*\mathcal E}(\Omega)$ are mapping from $^*\Omega$ to $^*\mathbb C$ (not from $\Omega$ to $^*\mathbb C$) and similarly for the rest of the spaces.

\item If $X\subseteq\mathbb R^d$, we let $\mu(X)=\{x+dx: x\in X,\, dx\in{^*\mathbb R^d}, dx\approx 0\}$ for the set of \emph{near-standard points of} $^*\!X$. Here $dx\approx 0$ means that $||dx||$ is an infinitesimal in $^*\mathbb R$. If $f\in{^*\mathcal E}(\Omega)$, then the following are equivalent: (a) $f\rest\mu(\Omega)=0$;  (b) $f\rest{^*K}=0$ for all $K\Subset\Omega$. 
\end{itemize}
\section{Non-Standard Version of Egorov Algebra $\widehat{^*\mathcal E}(\Omega)$}\label{S: Non-Standard Version of Egorov Algebra widehat*E(Omega)}

\begin{definition}[Non-Standard Version of Egorov Algebra]\label{D: Non-Standard Version of Egorov Algebra} Let $\Omega$ be an open set of $\mathbb R^d$. \begin{D-enum}
\item We let $
\widehat{^*\mathcal E}(\Omega)={^*\mathcal E(\Omega)}/\mathcal N(\Omega)$, where $\mathcal N(\Omega)=\{f\in{^*\mathcal E(\Omega)}: f\rest\mu(\Omega)=0\}$.
\item We supply $\widehat{^*\mathcal E}(\Omega)$ with the operations of a differential algebra over the field $^*\mathbb C$ with the operations inherited from $^*\mathcal E(\Omega)$.
\item For every $f\in{^*\mathcal E}(\Omega)$ we let $\widehat{f}=f+\mathcal N(\Omega)$ or $\widehat{f}=f\rest\mu(\Omega)$ and refer to $\widehat{f}$ as a \emph{generalized function on $\Omega$}. By exception, we shall write simply\; $c$\;  instead of\; $\widehat{c}$\; in the particular case $c\in{^*\mathbb C}$ (if $c$ is treated as a constant function in $^*\mathcal E(\Omega)$). If $\mathcal S\subseteq{^*\mathcal E(\Omega)}$, we let $\widehat{\mathcal S}=\{\widehat f: f\in S\}$. In particular, $\widehat{^*\mathcal E(\Omega)}= \widehat{^*\mathcal E}(\Omega)$,  $\widehat{\mathcal N(\Omega)}=\{0\}$, $\widehat{^*\mathbb C}={^*\mathbb C}$ and $\widehat{^*\mathbb R}={^*\mathbb R}$.
\item For every $\widehat{f}\in\widehat{\mathcal E}(\Omega)$ we define $\widehat{f}: \mu(\Omega)\mapsto{^*\mathbb C}$ by $\widehat{f}(\xi)={^*\!f}(\xi)$.
\item Let $\mathcal O$ be an open subset of $\Omega$ and $\widehat{f}\in\widehat{^*\mathcal E}(\Omega)$. We define the \emph{restriction}  $\widehat{f}\,\rest\mathcal O\in\widehat{^*\mathcal E}(\mathcal O)$ by  $\widehat{f}\,\rest\mathcal O= \widehat{f\rest{^*\mathcal O}}$. We say that $\widehat{f}$ \emph{vanishes on} $\mathcal O$ if $\widehat{f}\,\rest\mathcal O=0$ in $\widehat{^*\mathcal E}(\mathcal O)$. The \emph{support} $\supp(\widehat{f})$ of $\widehat{f}$ is the the complement to $\Omega$ of the largest open subset of $\Omega$, on which $\widehat{f}$ vanishes.
\item Let $X$ be a Lebesgue measurable subset of $\mathbb R^d$ whose closure is a compact subset of $\Omega$. We define a (Lebesgue) \emph{integral of $\widehat{f}\in\widehat{^*\mathcal E}(\Omega)$ over} $X$ with values in $^*\mathbb C$ by the formula $\int_X\widehat{f}(x)\, dx=\int_{^*X} f(\xi)\, d\xi$.
\item We define the \emph{pairing between} $\widehat{^*\mathcal E}(\Omega)$ and $\mathcal D(\Omega)$ by $\langle\widehat{f}, \varphi\rangle=\int_{^*\Omega} f(\xi)\;{^*\!\varphi(\xi)}\, d\xi$ for all $\widehat{f}\in\widehat{^*\mathcal E}(\Omega)$ and all $\varphi\in\mathcal D(\Omega)$, where $^*\!\varphi$ stands for the non-standard extension of $\varphi$.
\item We say that two generalized functions $\widehat{f},\; \widehat{g}\in\widehat{^*\mathcal E}(\Omega)$ are \emph{weakly equal} or \emph{associated}, and write $\widehat{f}\;\cong\;\widehat{g}$, if $\bra\widehat{f}, \varphi\ket=\bra\widehat{g}, \varphi\ket$ for all $\varphi\in\mathcal D(\Omega)$.
\item We define \emph{standard embedding} $\sigma: \mathcal E(\Omega)\to{^*\mathcal E(\Omega)}$ by $\sigma(f)=\widehat{^*\!f}$, where $\widehat{^*\!f}={^*\!f}\rest\mu(\Omega)$. 
\end{D-enum}
\end{definition}
\begin{theorem}[Basic Properties of $\widehat{^*\mathcal E}(\Omega)$]\label{T: Basic Properties of widehat*E(Omega)}  
\begin{T-enum}
\item $\widehat{^*\mathcal E}(\Omega)$ is a differential algebra over the field $^*\mathbb C$. Also, the mapping $f\, +\, \mathcal N(\Omega)\mapsto f\rest\mu(\Omega)$ from $^*\mathcal E(\Omega)/\mathcal N(\Omega)$ onto $\{ f\rest\mu(\Omega): f\in{^*\mathcal E(\Omega)}\}$  is a differential algebra isomorphism (justifying the notation $\widehat{f}=f\rest\mu(\Omega)$ used in advance).
\item $^*\mathbb C=\{\widehat f\in\widehat{^*\mathcal E}(\Omega): \nabla\widehat f=0\}$ for every open connected subset $\Omega$ of $\mathbb R^d$.
\item The family\, $\{\widehat{^*\mathcal E}(\Omega)\}_{\Omega\in\mathcal T^d}$, is a sheaf of differential algebras on $\mathbb R^d$, where $\mathcal T^d$ stands for the usual topology on $\mathbb R^d$.
\item   $\sigma[\mathcal E(\Omega)]$ is a differential $\mathbb C$-subalgebra of $\widehat{^*\mathcal E}(\Omega)$, isomorphic of $\mathcal E(\Omega)$. Also, $\int_X\sigma(f)(v)\, dx=\int_X f(x)\, dx$ for all $f\in\mathcal E(\Omega)$ and all Lebesgue measurable subset $X$ of $\mathbb R^d$ with compact closure in $\Omega$. Moreover, $\{\sigma[\mathcal E(\Omega)]\}_{\Omega\in\mathcal T^d}$ is a subsheaf of $\{\widehat{^*\mathcal E}(\Omega)\}_{\Omega\in\mathcal T^d}$. 
\end{T-enum}
\end{theorem}
\begin{proof}  For the proof we refer to (Todorov~\cite{tdTod96}, \S5) or/and  (Vernaeve~\cite{HVernPHD}).
\end{proof}

\section{Non-Standard Delta-Function}\label{S: Non-Standard Delta-Function}

	\quad We discuss the existence of a particular non-standard delta-function $\Delta$ in the space $^*\mathcal D(\mathbb R^d)$,  the non-standard extension $\mathcal D(\mathbb R^d)$. In this section we slightly modify similar results in Todorov~\cite{tdTod90}-\cite{tdTod92}.
\begin{lemma}[Non-Standard Delta-Function] \label{L: Non-Standard Delta-Function} For every $d\in\mathbb N$ there exists (not necessarily unique) $\Delta\in{^*\mathcal D}(\mathbb R^d)$ such that $\Delta(\xi)=0$ for all infinitely large and all finite, but non-infinitesimal $\xi\in{^*\mathbb R^d}$ and such that $
\int_{^*\mathbb R^d}\Delta(\xi)\, {^*\!\varphi}(\xi)\, d\xi= \int_{^*\mathbb R^d}\Delta(-\xi)\, {^*\!\varphi}(\xi)\, d\xi=\varphi(0)$ for all continuous functions $\varphi\in\mathcal C(\mathbb R^d)$. We let $\rho={^*\!\sup}\big\{||\xi||: \xi\in{^*\mathbb R^d},\, \Delta(\xi)\not=0\big\}$ for the radius of support of $\Delta$. Moreover, for each open $\Omega\subseteq\mathbb R^d$ and each $\varepsilon\in{^*\mathbb R_+}$ we let 
\[
\Omega_\varepsilon=\{\xi\in{^*\Omega}: {\rm dist}(\xi, \partial\Omega)\geq\varepsilon\;\, \&\;\, {\rm dist}(\xi, 0) \leq 1/\varepsilon\},
\]
and define $\Pi_\Omega: {^*\mathbb R^d}\mapsto{^*\mathbb C}$ by the formula $
\Pi_\Omega(\xi)= \int_{\Omega_{3\rho}}\Delta(\xi-\eta)\, d\eta$.
\end{lemma}
\begin{theorem}[Regularization in $^*\mathcal E(\Omega)$]\label{T: Regularization in *E(Omega)}   Let  $\Omega$ and $\Delta$ be as in Lemma~\ref{L: Non-Standard Delta-Function}. Then:
\begin{T-enum}
\item ${^*\!f}\star\Delta\in{^*\mathcal E}(\mathbb R^d)$ for every $f\in\mathcal C(\mathbb R^d)$. Here ${^*\!f}\star\Delta: {^*\mathbb R^d}\to{^*\mathbb C}$ is defined by $({^*\!f}\star\Delta)(\xi)=\int_{^*\mathbb R^d}\,{^*\!f(\eta)}\,\Delta(\xi-\eta)\, d\eta$. Moreover,  ${^*\!f}\star\Delta$ is an extension of $f$ from $\mathbb R^d$ to $^*\mathbb R^d$, in symbol, $({^*\!f}\star\Delta)\,\rest\,\mathbb R^d=f$.
\item $^*P\star\Delta={^*P}$ for all polynomials $P\in\mathbb C[\Omega]$.
\item $\rho$ is a positive infinitesimal in $^*\mathbb R$ and $\mu(\Omega)\subseteq\Omega_{2\rho}\subset\Omega_\rho\subseteq{^*\Omega}$.
\item  $\Pi_\Omega\in{^*\mathcal D(\Omega)}$, $\Pi_\Omega\rest\mu(\Omega)=1$. Moreover,  $\supp(\Pi_\Omega)=\Omega_{2\rho}$.
\item $\Pi_\Omega({^*T}\star\Delta)\in{^*\mathcal D(\Omega)}$ for all $T\in\mathcal D^\prime(\Omega)$. 
\end{T-enum}
\end{theorem}
\begin{proof} For (i) and (ii) we refer to (Todorov~\cite{tdTod92}, where the results are based on infinite-dimensional linear algebra and saturation principle in non-standard analysis (Lindstr\o m~\cite{tLindstrom} and/or Loeb\,\&Wolff~\cite{LoebWolff}). The fact that $\rho$ is an infinitesimal follows by underflow principle (Lindstr\o m~\cite{tLindstrom}). For the standard counterpart of the rest we refer to  (Vladimirov~\cite{vVladimirov}, \S 4.6).
\end{proof}
\section{Embedding of Distributions  into $\widehat{^*\mathcal E}(\Omega)$}\label{S: Embedding of Distributions into widehat*E(Omega)}

	\quad Although the mapping $T\to\Pi_\Omega({^*T}\star\Delta)$ from $\mathcal D^\prime(\Omega)$ to $^*\mathcal E(\Omega)$ (Theorem~\ref{T: Regularization in *E(Omega)}) is injective, it does not commute with the partial derivatives $\partial^\alpha$ in $^*\mathcal E(\Omega)$. Moreover, the family $\{{^*\mathcal E}(\Omega)\}_{\Omega\in\mathcal T^d}$ is not a sheaf on $\mathbb R^d$. Thus $^*\mathcal E(\Omega)$ cannot be treated as an algebra of generalized functions on $\mathbb R^d$ we are looking for; we return to the algebra $\widehat{^*\mathcal E}(\Omega)$ defined in Section~\ref{S: Non-Standard Version of Egorov Algebra widehat*E(Omega)}.
	
\begin{definition}[Embedding of Distributions in $\widehat{^*\mathcal E}(\Omega)$]\label{D: Embedding of Distributions in widehat*E(Omega)} Let  $\Omega$ and $\Delta$ be chosen (and fixed)  as in Lemma~\ref{L: Non-Standard Delta-Function}. We define $\iota_\Omega:\mathcal D^\prime(\Omega)\mapsto\widehat{^*\mathcal E}(\Omega)$ by $\iota_\Omega(T)=\widehat{\Pi_\Omega({^*T}\star\Delta)}$ or equivalently, by $\iota_\Omega(T)=\Pi_\Omega({^*T}\star\Delta)\rest\mu(\Omega)$. Here $^*T\star\Delta: {^*\Omega}\to{^*\mathbb C}$ is defined by $(^*T\star\Delta)(\xi)=\langle{^*T}(\eta), \Delta(\xi-\eta)\rangle$ on the ground of transfer principle (Lindstr\o m~\cite{tLindstrom} and/or Loeb\,\&Wolff~\cite{LoebWolff}).
\end{definition}
\begin{theorem}[Properties of the Embedding]\label{T: Properties of the Embedding} 
\begin{T-enum}
 \item  $\iota_\Omega$ commutes with the partial derivatives on $\mathcal D^\prime(\Omega)$.  Moreover, $\langle\iota_\Omega(T), \varphi\rangle=\langle T, \varphi\rangle$ for all $\varphi\in\mathcal D(\Omega)$ (Definition~\ref{D: Non-Standard Version of Egorov Algebra}). Consequently, $\iota_\Omega$ is injective and  $\iota_\Omega[\mathcal D^\prime(\Omega)]$ is a differential $\mathbb C$-vector subspace of\,  
$\widehat{^*\mathcal E}(\Omega)$.
 \item $(\iota_\Omega\circ S_\Omega)[\mathcal D(\Omega)]$, $(\iota_\Omega\circ S_\Omega)[\mathcal E(\Omega)]$,  $(\iota_\Omega\circ S_\Omega)[\mathcal C(\Omega)]$ and $(\iota_\Omega\circ S_\Omega)[\mathcal L_{loc}(\Omega)]$ are $\mathbb C$-vector subspaces of $\widehat{^*\mathcal E}(\Omega)$. Moreover, $(\iota_\Omega\circ S_\Omega)[\mathcal D(\Omega)]$ and $(\iota_\Omega\circ S_\Omega)[\mathcal E(\Omega)]$ are differential $\mathbb C$-vector subspaces of\, $\widehat{^*\mathcal E}(\Omega)$. Also, we have $(\iota_\Omega\circ S_\Omega)(f)\cong\sigma(f)$ for all $f\in\mathcal E(\Omega)$  (Definition~\ref{D: Non-Standard Version of Egorov Algebra}).
  \item $(\iota_\Omega\circ S_\Omega)(P)=\sigma(P)$ for all polynomials $P\in\mathbb C[\Omega]$. Consequently, $(\iota_\Omega\circ S_\Omega)\big[\mathbb C[\Omega]\big]$ is a differential subring (a differential $\mathbb C$-subalgebra) of\,  $\widehat{^*\mathcal E}(\Omega)$, which is isomorphic to $ \mathbb C[\Omega]$. We summarize these in the chain of embeddings: $ \mathbb C[\Omega]\subset\mathcal D^\prime(\Omega)\subset\widehat{^*\mathcal E}(\Omega)$, after dropping $\iota_\Omega$.
 \item The family $\big\{\iota_\Omega[\mathcal D^\prime(\Omega)]\big\}_{\Omega\in\mathcal T^d}$ is subsheaf of $\{\widehat{^*\mathcal E}(\Omega)\}_{\Omega\in\mathcal T^d}$ on $\mathbb R^d$. Consequently, $\supp(T)=\supp(\iota_\Omega(T))$ for all $T\in\mathcal D^\prime(\Omega)$. 
\item  Let $f\in\mathcal C(\Omega)$ be continuous function. Then $(\iota_\Omega\circ S_\Omega)(f)$ is an extension of $f$ (from $\Omega$ to $\mu(\Omega$), i.e. $(\iota_\Omega\circ S_\Omega)(f)(x)=f(x)$ for all $x\in\Omega$. In particular,  $\partial^\alpha\,(\iota_\Omega\circ S_\Omega)(f)(x)=\partial^\alpha f(x)$ for all $f\in\mathcal E(\Omega)$, all $\alpha\in\mathbb N_0^d$ and all $x\in\Omega$. 
\item  Let $X$ and $Y$ be two open subsets of $\mathbb R^d$ and $\theta\in{\rm Diff}(X, Y)$. Let the mapping $T\to T(\theta)$, from $\mathcal D^\prime(X)$ to $\mathcal D^\prime(Y)$, stands for the change of variables in the sense of distribution theory (H\"{o}rmander~\cite{lHormander90a}, \S 6.3-\S 6.4) and (Vladimirov~\cite{vVladimirov}, p.26). We define $\theta_*: \widehat{^*\mathcal E}(X)\to \widehat{^*\mathcal E}(Y)$ by $\theta_*(\widehat f)=(f\circ{^*\theta^{-1}})\rest\mu(Y)$, where $^*\theta^{-1}$ stands for the non-standard extension of $\theta^{-1}$. Then $\theta_*(\iota_X(T))\cong\iota_Y(T(\theta))$ for all $T\in\mathcal D^\prime(X)$.
\end{T-enum}
\end{theorem}
\begin{proof} Relatively straightforward  consequences from Theorem~\ref{T: Regularization in *E(Omega)}.
\end{proof}
\begin{examples}[Some Particular Generalized Functions in $\widehat{^*\mathcal E}(\Omega)$] \label{Exs: Some  Generalized Functions in widehat*E(Omega)} 
\begin{Exs-enum}
\item $\iota_{\mathbb R^d}(\delta)=\widehat\Delta$ and more generally,  $\iota_{\mathbb R^d}(\partial^\alpha\delta)=\widehat{\partial^\alpha}\Delta$ for all $\alpha\in\mathbb N_0^d$. We write these more casually as $\partial^\alpha\delta\in\widehat{^*\mathcal E}(\mathbb R^d)$.
\item $(\widehat\Delta)^n\in\widehat{^*\mathcal E}(\mathbb R^d)$ for all $n\in\mathbb N$, since $\widehat{^*\mathcal E}(\mathbb R^d)$ is an algebra. We write this more casually as $\delta^n\in\widehat{^*\mathcal E}(\mathbb R^d)$.
\item Let $f: \mathbb C\mapsto\mathbb C$ stand for $f(z)=e^z$. Clearly, $e^{x+iy}\in\mathcal E(\mathbb R^2)$. Thus $e^\Delta\in{^*\mathcal E}(\mathbb R^2)$ (we skip the asterisk in front of $^*e^z$) and $\widehat{e^\Delta}\in\widehat{^*\mathcal E}(\mathbb R^2)$. We write more causally, $e^\delta\in\widehat{^*\mathcal E}(\mathbb R^2)$. Notice that $e^\delta$ makes sense as well in (Egorov~\cite{yEgorov90a}), but not in  Colombeau algebra,  since $e^{x+iy}$ is \emph{non-moderate} in the variable $x$ (Colombeau~\cite{jCol84a}). 
\item Let temporarily write $\Delta_d$ instead of $\Delta$ indicating that $\Delta_d\in{^*\mathcal D}(\mathbb R^d)$. Let $(e_1, e_2, \dots, e_d)$ be the standard basis for $\mathbb R^d$. Then $\widehat{\Delta}_d\rest{\rm span}(e_n)\cong\widehat {\Delta}_1$, where $\widehat{\Delta}_d\rest{\rm span}(e_n):=\Delta_d\rest\mu({\rm span}(e_n))$. Notice that ${\rm span}(e_n)$ is a \emph{smooth submanifold} (not open subset) of $\mathbb R^d$.
\end{Exs-enum}
\end{examples}
	Very much like in Colombeau and Egorov theories, Schwartz distribution can be multiplied within $\widehat{^*\mathcal E}(\Omega)$, since the latter is a differential (commutative and associative) algebra. How good (or bad) is this product? In a lack of compelling applications to other branches of mathematics or physics, we making our judgement mostly by applying this product to the $\iota_\Omega$-images in $\widehat{^*\mathcal E}(\Omega)$ of the classical functions. Here is our test: 
\begin{corollary}[Multiplication of Classical Functions]\label{C: Multiplication of Classical Functions}
\begin{C-enum}
\item The product in the algebra  $\widehat{^*\mathcal E}(\Omega)$, if restricted to  $\mathbb C[\Omega]$ (more precisely, on $(\iota_\Omega\circ S_\Omega)\big[\mathbb C[\Omega]\big]$), coincides with the usual product between polynomials, i.e. for every $P, Q\in\mathbb C[\Omega]$ we have
\[
(\iota_\Omega\circ S_\Omega)(PQ)=(\iota_\Omega\circ S_\Omega)(P)\cdot(\iota_\Omega\circ S_\Omega)(Q).
\]
\item $(\iota_\Omega\circ S_\Omega)\big[\mathbb C[\Omega]\big]={\sigma}\big[\mathbb C[\Omega]\big]={\sigma}[\mathcal E(\Omega)]\cap(\iota_\Omega\circ S_\Omega)[\mathcal E(\Omega)]$.
\item  For every two continuous functions $f, g\in\mathcal C(\Omega)$ and for all (standard) $x\in\Omega$ we have $(\iota_\Omega\circ S_\Omega)(fg)(x)=(\iota_\Omega\circ S_\Omega)(f)(x)\cdot(\iota_\Omega\circ S_\Omega)(g)(x)=f(x)g(x)$.
\item $\iota_\Omega(fT)\cong\sigma(f)\, \iota_\Omega(T)$ (Definition~\ref{D: Non-Standard Version of Egorov Algebra}) for all $f\in\mathcal E(\Omega)$ and all $T\in\mathcal D^\prime(\Omega)$, where the product $fT$ is in the sense of distribution theory (Vladimirov~\cite{vVladimirov}, \S 1.10). In particular, $\iota_\Omega(P\,T)\,\cong\,(\iota_\Omega\circ S_\Omega)(P)\; \iota_\Omega(T)$ for all polynomials $P\in\mathbb C[\Omega]$ and all $T\in\mathcal D^\prime(\Omega)$.
\end{C-enum}
\end{corollary}
\section{Regular Algebra}\label{S: Regular Algebra}

	\quad The algebra $\widehat{\mathcal R}_\rho(\Omega)$ defined below is similar, but different from the algbra $\mathcal G^\infty(\Omega)$ introduced and study in (Oberguggenberger~\cite{mOber92},\,\cite{mOber06}). We should mention that $\widehat{\mathcal R}_\rho$  does not contain the counterexample in Vernaeve~\cite{HVern21}).
		
\begin{definition}[Regular Algebra]\label{D: Regular Algebra} Let  $\Omega$ and $\Delta$ be chosen (and fixed)  as in Lemma~\ref{L: Non-Standard Delta-Function}. Let  ${^\sigma\!\mathcal E}(\Omega)=\{^*\!f: f\in\mathcal E(\Omega)\}$ and $\mathcal M_\rho=\{\xi\in{^*\mathbb C}: |\xi|\leq\rho^{-n} \text{ for some } n\in\mathbb N\}$. Let $\mathcal R_\rho(\Omega)$ denote the subring of $^*\mathcal E(\Omega)$ generated by ${^\sigma\!\mathcal E}(\Omega)\cup\mathcal M_\rho$, in symbol, $\mathcal R_\rho(\Omega)={^\sigma\!\mathcal E}(\Omega)(\mathcal M_\rho)$. The \emph{algebra of $\rho$-regular functions} is defined by  $\widehat{\mathcal R}_\rho(\Omega)=\{\widehat{f}:  f\in\mathcal R_\rho(\Omega)\}$ or equivalently, by $\widehat{\mathcal R}_\rho(\Omega)=\{f\rest\mu(\Omega): f\in\mathcal R_\rho(\Omega)\}$.
\end{definition}
\begin{theorem}  Under the assumption of the above definition we have:
\begin{T-enum}
\item $\widehat{\mathcal R}_\rho$ is a differential $\mathbb C$-subalgebra of\,  $\widehat{^*\!\mathcal E}(\Omega)$. 
\item If $\widehat f\in\widehat{\mathcal R}_\rho(\Omega)$, then $(\forall \xi\in\mu(\Omega))(\exists n\in\mathbb N)(\forall\alpha\in\mathbb N_0^d)(|\partial^\alpha \widehat{f}(\xi)|\leq\rho^{-n})$.
\item $\widehat{\mathcal R}_\rho(\Omega)\cap{\iota_\Omega}[\mathcal D^\prime(\Omega)]=({\iota_\Omega}\circ S_\Omega)[\mathbb C[\Omega]]$.
\end{T-enum}
\end{theorem}
\begin{proof} We leave the verification to the reader.
\end{proof}


\begin{thebibliography}{99}

\bibitem{jCol84a} J. F. Colombeau, \emph{New Generalized Functions and Multiplication of Distributions}, Math. Studies 84, North-Holland-Amsterdam-New York-Oxford, 1984.

\bibitem{yEgorov90a} Yu. V. Egorov, {\em A contribution to the theory of generalized functions} (Russian). Uspekhi Mat Nauk (Moskva)
45, No 5 (1990), 3-40. \newline
\emph{English translation} in Russian Math. Surveys {\bf 45}:5  (1990), p. 1-49. 

\bibitem{lHormander90a} Lars H\"{o}rmander, {\em The Analysis of Linear Partial Differential Operators I}, volume \textbf{256} of Grundlehren der mathematischen Wissenschaften, Springer, Berlin, 1990.

\bibitem{tLindstrom} Tom Lindstr\o m, {\em An invitation to nonstandard analysis}, in: Nonstandard Analysis and its Applications, N. Cutland (Ed), Cambridge U. Press, 1988, p.~1-105.

\bibitem{LoebWolff} Peter A. Loeb and Manfred Wolff (Eds.), \emph{Nonstandard Analysis for the Working Mathematician}, Kluwer Academic Publishers, Dordrecht/Boston/London, 2000.

\bibitem{mOber92} M. Oberguggenberger, \emph{Multiplication of Distributions and Applications to Partial Differential Equations}, Pitman Research Notes Math., 259, Longman, Harlow, 1992. 

\bibitem{OberTod98} M. Oberguggenberger and T. Todorov, \emph{An Embedding of  Schwartz Distributions in the Algebra of Asymptotic Functions}, International J. Math. \& Math. Sci., Vol. 21, No. 3 (1998), p.417-428. 

\bibitem{mOber06} M. Oberguggenberger, \emph{Regularity Theory in Colombeau Algebras}, in Bulletin T. CXXXIII de l'Academie serbe des sciences et des arts, Classe des sciences math\'{e}matiques et naturelles, Sciences math\'{e}matique No 31, 2006, p. 147-162.
\bibitem{tdTod90}  Todor Todorov, {\em A Nonstandard Delta Function}, In: \emph{Proceedings of the American Mathematical Society}, Vol. 110,  Number 4, 1990, p.~1143--1144.

\bibitem{tdTod92} Todor Todorov, {\em Pointwise Kernels of Schwartz Distributions}, In: \emph{Proceedings of the American Mathematical Society}, Vol. 114,  No 3,  March 1992, p.~817--819.

\bibitem{tdTod96} Todor Todorov, {\em An Existence of Solutions for Linear PDE with $C^{\infty}$-Coefficients in an Algebra of Generalized Functions}, In: \emph{Transactions of the American Mathematical Society}, Vol. 348, 2, Feb. 1996, p.~ 673--689 \href{http://www.ams.org/journals/tran/1996-348-02}{\textcolor{blue}{\tt http://www.ams.org/journals/tran/1996-348-02}}.

\bibitem{TodVern08} Todor Todorov and Hans Vernaeve, {\em Full Algebra of Generalized Functions and Non-Standard  Asymptotic Analysis}, In: \emph{Logic and Analysis}, Springer, Vol. 1, Issue 3, 2008, pp. 205-234 (available at:
(http://www.logicandanalysis.com/index.php/jla/article/view/193/79) and/or at: \href{http://arxiv.org/abs/0712.2603}{\textcolor{blue}{\tt http://arxiv.org/abs/0712.2603}}).

\bibitem{HVernPHD}	Hans Vernaeve, \emph{Nonstandard Contributions to the Theory of Generalized Functions}, Ph.D. Thesis, University of Gent, 2002.

\bibitem{HVern03} Hans Vernaeve, \emph{Optimal Embeddings of Distributions into Algebras}, Proceedings of the Edinburgh Mathematical Society (2003) \textbf{46}, p. 373-378.

\bibitem{HVern21} Hans Vernaeve, \emph{Regularity of nonlinear generalized functions: A Counterexample in the nonstandard setting}, in Examples and Counterexamples \textbf{1} (2021) 100021.

\bibitem{vVladimirov} V. S. Vladimirov, {\em Generalized Functions in Mathematical Physics}, Mir-Publisher, Moscow, 1979.
\end{thebibliography}
\end{document}